\documentclass[11 pt,a4paper]{article}
\usepackage{amsmath,amscd}
\usepackage{amstext}
\usepackage{amssymb,epsfig}
\usepackage{amsthm}
\usepackage{amsfonts}
\usepackage{graphicx}
\usepackage{tikz}
\usepackage[all]{xy}

\renewcommand{\footnote}{\endnote}
\newtheorem{theorem}{Theorem}
\newtheorem{lemma}{Lemma}
\newtheorem{proposition}{Proposition}
\newtheorem{definition}{Definition}
\newtheorem{remark}{Remark}
\newtheorem{corollary}{Corollary}
\newtheorem{example}{Example}

\usepackage[utf8]{inputenc}      

\begin{document}
\title{A characterization of quasi-rational polygons}
\author{Nicolas Bedaride\footnote{ Laboratoire d'Analyse Topologie et Probabilités  UMR 7353, Université Aix Marseille, avenue escadrille Normandie Niemen 13397 Marseille cedex 20, France. nicolas.bedaride@univ-amu.fr}}
\date{}
\maketitle
\begin{abstract}
The aim of this paper is to study quasi-rational polygons related to the outer billiard. We compare different notions introduced in \cite{Gu.Si.92} and \cite{Sch.09} and make a synthesis of those. 

\end{abstract}

\section{Introduction}
The outer billiard map is a transformation $T$ of the exterior of a planar convex bounded domain $D$ defined as follows: $T(M)=N$ if the segment $MN$ is tangent to the boundary of $D$ at its midpoint, and $D$ lies at the right of $MN$. The outer billiard map is not defined if the tangent segment $MN$ shares more than one point with the boundary of $D$. In the case where $P$ is a convex polygon; the set of points for which $T$ or any of its iterations is not defined is contained in a countable union of lines and has zero measure. The dual billiard map has been introduced by Neumann in \cite{Neu.59} as a toy model for the planet orbits. One of the most interesting questions was whether the orbits of $T$ might escape to infinity for a polygonal domain $D$.  
 
Two particular classes of polygons have been 
introduced by Kolodziej et al. in several articles, see \cite{Kol,Gu.Si.92, Viv.Shai.87}. These classes are named rational and quasi-rational polygons and contain all the regular polygons. A rational polygon has vertices on a lattice of $\mathbb{R}^2$. They prove that every orbit outside a polygon in this class is bounded. Every regular polygon is a quasi-rational polygon, and it is not a rational polygon except if there are $3, 4$ or $6$ edges. In the case of the regular pentagon, Tabachnikov completely described the dynamics of the outer billiard map in terms of symbolic dynamics, see \cite{Ta.95}. He proves that some orbits are bounded and non periodic. The symbolic coding of this map has been given in \cite{Bed.Ca.10} for a regular polygon with $3, 4, 5, 6$ and $10$ edges. 

For non quasi-rational polygons, there is no general study. The case of trapezoids has been studied. The set of trapezoids can be parametrized up to affinity by one parameter.  For an irrational parameter, it is not a quasi-rational polygon, and the proof of \cite{Gu.Si.92} cannot be used for a polygon with parallel sides. Nevertheless, Li proved that all the orbits of the outer billiard map are bounded (this theorem is also proved by Genin) see \cite{Li.09} and \cite{Gen.08}.
Recently Schwartz described a family of quadrilaterals, named kites, for which there exists unbounded orbits, see \cite{Sch.07} and \cite{Sch.09}. In these papers Schwartz introduces many tools in order to study the dynamics.  These tools can also be used in the case of regular polygons, see \cite{Sch.10}.

In this article we investigate the case of quasi-rational polygons. The main 
achievements of the paper consist of  a synthesis of results of \cite{Gu.Si.92} 
and the notions introduced by Schwartz. These links allow us to give a new characterization of this class and to give some simple conditions which guarantee the quasi rationality.

\begin{remark}
In this article, $P$ is a polygon with $n$ vertices without parallel edges, see last section for some comments. 
All the figures correspond to the same polygon.
\end{remark}


\section{Overview of the paper}

First we recall usual definitions about dual billiard in Section \ref{secdual} and introduce our definition of quasi-rational polygon. Next, in Section \ref{sec:equivalence}, we show that our definition is equivalent to the old one of  \cite{Gu.Si.92} and also similar to \cite{Sch.09}. In Section \ref{sec:gs} we prove the classical theorem on quasi-rational polygon using our definition. Finally in Section \ref{sec:new} we use our definition to obtain new results on quasi-rational polygons.


\section{Outer billiard}\label{secdual}
We refer to \cite{Ta} or \cite{Gu.Si.92}.
We consider a convex polygon $P$ in $\mathbb{R}^2$ with $n$ vertices. Let $\overline{P}=\mathbb{R}^2\setminus P$ be the complement of $P$.\\
We fix an orientation on $\mathbb{R}^2$. We will define the outer billiard map off a countable union of lines. The map will be defined for all time.

For a point $M\in\overline{P}$, there are two half-lines $R,R'$ emanating from $M$ and tangent to $P$, see Figure \ref{fig-def-1}. Assume that the oriented angle $R,R'$ has positive measure. Denote by $A^+,A^-$ the tangent points on $R$ respectively $R'$.
We say that $A^+$ is the vertex {\bf associated} to $M$.
\begin{definition}\label{point}
The outer billiard map is the map $T$ defined as follows:
$$T(M)=r_{A^+}(M)$$ 
where $r_{A^+}$ is the reflection about $A^+$.
\end{definition}

\begin{figure}
\begin{center}
\begin{tikzpicture}
\fill (0,0)--(2,0)--(1,2)--(0,1)--cycle;
\draw[dashed] (5,2)--(-1,2);
\draw[dashed] (-1,2)--(1,-2);
\draw (3,2) node{$\bullet$} node[above]{$TM$};
\draw (-1,2) node{$\bullet$} node[above]{$M$};
\draw (1,2) node[above]{$A^+$};
\draw (0,0) node[left]{$A^-$};
\draw (1,-2) node{$\bullet$} node[left]{$T^{-1}M$};
\draw[dashed](3,2)--(1.5,-1);
\draw (4,2) node[above]{$R$};
\end{tikzpicture}
\begin{tikzpicture}[scale=.5]
\fill (0,0)--(2,0)--(1,2)--(0,1)--cycle;
\draw[dashed] (0,0)--++(0,-4);
\draw[dashed] (2,0)--++(4,0);
\draw[dashed] (1,2)--++(-2,4);
\draw[dashed] (0,1)--++(-3,-3);
\end{tikzpicture}

\end{center}
\caption{The outer billiard map}\label{fig-def-1}
\end{figure}
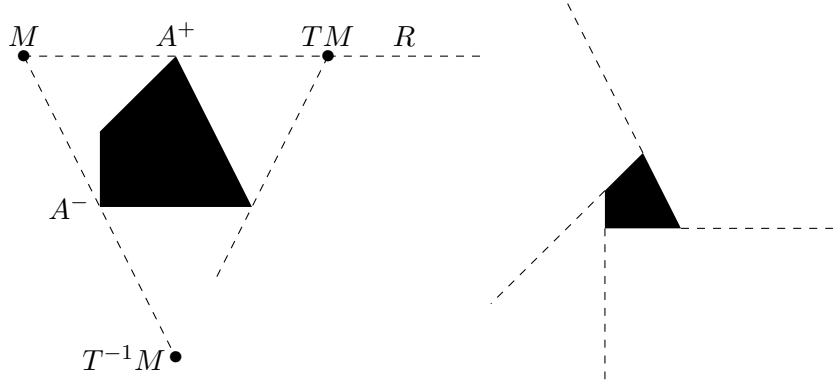

\begin{definition}
A polygon $P$ is said to be rational if the vertices of $P$ are on a lattice of $\mathbb{R}^2$.
\end{definition}

We refer to \cite{Sch.09}.
Consider a polygon without parallel edges. Assume the edges are oriented counterclockwise sense (while $T$ is oriented clockwise). For each edge, consider the vertex of $P$ furthest from the line supporting the edge. It is unique by convexity and assumption. Then denote by $V$ the vector equal to twice the vector between the final vertex of the edge and this vertex. A {\bf strip} is the band formed by the line supporting the edge and $V$, see Figure \ref{vector}. It is denoted $(\Sigma,V)$ or $\Sigma$ if no confusion is possible. We index them with respect to the slopes of the sides of the polygon, this gives the sequences $(\Sigma_i,V_i)_{1\leq i\leq n}$.


\begin{definition}
Let $\alpha_1,\dots, \alpha_n$ be non zero real numbers, we say that $(\alpha_1,\dots,\alpha_n)$ are commensurate if $\frac{\alpha_2}{\alpha_1},\dots,\frac{\alpha_n}{\alpha_{n-1}},\frac{\alpha_1}{\alpha_n}$ are rational numbers.
\end{definition}

\begin{definition}
The polygon $P$ is quasi-rational if and only if\\
$(|V_1\wedge V_{2}|,\dots ,|V_n\wedge V_1|)$ are commensurates.
\end{definition}

For example, consider the polygon with vertices $A,B, C, D$, see Figure \ref{vector}. The vectors are equal to:
$V_1=2\vec{CB}, V_2=2\vec{AC}, V_3=2\vec{BD}, V_4=2\vec{BA}$.

\begin{figure}[h!]
\begin{center}
\begin{tikzpicture}[scale=.5]

\fill (0,0)--(2,0)--(1,2)--(0,1)--cycle;
\draw (0,1) node[left]{$C$};
\draw (0,0) node[below]{$A$};
\draw (2,0) node[below]{$B$};
\draw (1,2) node[left]{$D$};

\draw[dashed] (8,0)--++(-2,4);
\draw[dashed] (9,4)--++(-4,2);
\draw[dashed] (4,5)--++(-4,0);
\draw[dashed](-1,6)--++(-2,-4);
\draw (-6,0)--(8,0);
\draw (-6,4)--(8,4);
\draw (2,0)--++(-4,8);
\draw (2,0)--++(3,-6);
\draw (0,1)--++(6,6);
\draw (0,1)--++(-4,-4);
\draw (5,0)--++(4,4);
\draw (5,0)--++(-7,-7);
\draw (-2,0)--++(-4,8);
\draw (-2,0)--++(5,-10);

\draw (0,-8)--(0,6);
\draw (4,-8)--(4,6);

\end{tikzpicture}
\end{center}
\caption{Polygon and strips}\label{vector}
\end{figure}
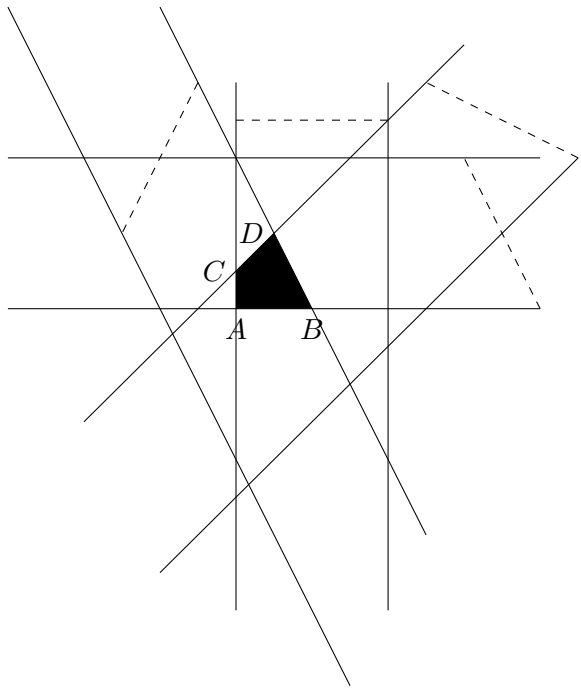


\section{Unfolding}
In this Section we recall the notion of unfolding introduced in \cite{Gu.Si.92}. This notion is used to transform the outer billiard map in a piecewise translation map defined on cones. 
\subsection{Definitions}

\begin{figure}[t]
\begin{center}
\begin{tikzpicture}
\draw[dashed] (0,0)--(2,0);
\draw[dashed] (0,0)--(2,2);
\draw[dashed] (0,0)--(0,2);
\draw[dashed] (0,0)--(-2,2);
\draw[dashed] (0,0)--(-2,0);
\draw[dashed] (0,0)--(-2,-2);
\draw[dashed] (0,0)--(0,-2);
\draw[dashed] (0,0)--(2,-2);

\draw (1.5,0) node[below]{$l_1$};
\draw (0,0) node[below]{$M$};
\draw (0,1.5) node[right]{$l_3$};

\draw (1,0.25)--(2,0.25);
\draw (2,0.25)--(1.5,1.25); 
\draw (1.5,1.25)--(1,0.75);
\draw  (1,0.75)--(1,0.25);

\draw (3,0.25)--(2,0.25);
\draw (3,0.25)--(3,-0.25);
\draw (2,0.25)--(2.5,-0.75);
\draw (2.5,-0.75)--(3,-0.25);

\draw (1,2.25)--(1.5,1.25); 
\draw (1.5,1.25)--(2,1.75);
\draw (2,1.75)--(2,2.25);
\draw (2,2.25)--(1,2.25);
\end{tikzpicture}
\begin{tikzpicture}
\draw (1,0.25)--(2,0.25);
\draw (2,0.25)--(1.5,1.25); 
\draw (1.5,1.25)--(1,0.75);
\draw  (1,0.75)--(1,0.25);
\draw (1.5,1.25) node[left]{$A'$};
\draw (3,0.25)--(2,0.25);
\draw (3,0.25)--(3,-0.25);
\draw (2,0.25)--(2.5,-0.75);
\draw (2.5,-0.75)--(3,-0.25);
\draw (2,.25) node[below left]{$A$};
\draw (0,0) node[below]{$M$};
\end{tikzpicture}
\end{center}
\caption{Necklace dynamics}\label{fig:neck}
\end{figure}
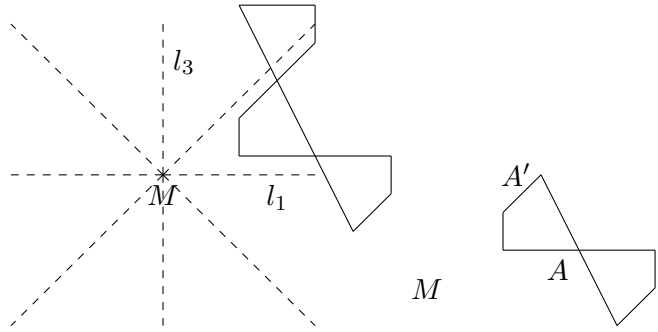

Consider two vertices $A, B$ of $P$, and the images of $P$ by the rotations $r_{A}, r_B$ of angle $\pi$. They are equal up to translation. Denote by $\tilde P$ one of these polygons. Let $\mathcal{S}_P$ be the set of those polygons in $\mathbb{R}^2$ that are images by a translation by $P$ or $\tilde P$. 
Let $M$ be a point in $\overline P$, define the following bijective map. 
$$\begin{array}{ccccc}
\pi_M&:&\mathcal{S}_P&\to& \mathbb{R}^2\times\{-1,1\} \\
&&Q&\mapsto&(A,\varepsilon)\end{array}$$ 

Consider the image of $M$ by the outer billiard map outside $Q$. It is obtained by a rotation of angle $\pi$ centered at a vertex of $Q$. Let $A$ be this vertex of $Q$. Moreover we take $\varepsilon=1$ if $Q$ is a translate of $P$, $\varepsilon=-1$ if $Q$ is a translate of $\tilde P$. We say $A$ is associated to $M$ for $Q$. It is clear that $\pi_M$ is a bijection.
  
We define a new map called {\bf the unfolding} of the dual billiard map.

$$\begin{array}{ccccc}
\tilde T&:&\mathbb{R}^2\times\{-1,1\}&\to&\mathbb{R}^2\times\{-1,1\}\\
&&(A,\varepsilon)&\mapsto&(A',-\varepsilon)
\end{array}$$
The ordered pair $(A,\varepsilon)$ comes from a polygon $Q$ via the map $\pi_M$. Consider the polygon $Q'$ image of the polygon $Q$ by a rotation of angle $\pi$ of center $A$, see Figure \ref{fig:neck}. 
The point $A'$ is the vertex associated to $M$ for $Q'$.
 
The dynamics of this map is related to the outer billiard map by the following result. In what follows we will also denote by $\tilde T$ the projection of $\tilde T$ to $\mathbb{R}^2$.

\begin{definition}
Denote by $(l_i)_{i\leq n}$ the lines passing through $M$ and parallel to the edges of $P$. They defines $2n$ cones $(C_i)_{1\leq i\leq 2n}$, each cone has for boundary two half lines $R_i, R_{i+1}$. 
\end{definition}

\begin{proposition}\cite{Gu.Si.92}\label{lem:closed}
We have: 
\begin{itemize}
\item The sequence $(T^k(M))_k$ is bounded (resp. periodic) if and only if there exists a point $Q\in \mathbb{R}^2\times\{-1,1\}$ such that the orbit of $Q$ is bounded (resp. periodic) for $\tilde T$.
\item For every cone $C_i$, there exists a vector $a_i$ such that if $A,\tilde T A\in C_i$, then the restriction of $\tilde T$ to a cone is a translation of vector $a_i$. Moreover we have for every integer $i$, $a_{n+i}=-a_i$. 
\item There exists a polygon $P^*$ with $2n$ edges, with vertices on $C_1,\dots,C_n$ such that each side is parallel to some $a_i$.
\end{itemize}
\end{proposition}

The sides of $P^*$ will be denoted $v_i^*, i=1\dots 2n$.

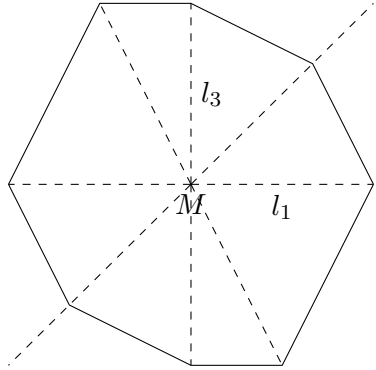
\begin{figure}[]
\begin{center}
\begin{tikzpicture}[scale=.4]
\draw[dashed] (0,0)--(6,0);
\draw[dashed] (0,0)--(6,6);
\draw[dashed] (0,0)--(0,6);
\draw[dashed] (0,0)--(-3,6);
\draw[dashed] (0,0)--(-6,0);
\draw[dashed] (0,0)--(-6,-6);
\draw[dashed] (0,0)--(0,-6);
\draw[dashed] (0,0)--(3,-6);

\draw (3,0) node[below]{$l_1$};
\draw (0,0) node[below]{$M$};
\draw (0,3) node[right]{$l_3$};

\draw (6,0)--++(-2,4)--++(-4,2)--++(-3,0)--++(-3,-6);
\draw (-6,0)--++(2,-4)--++(4,-2)--++(3,0)--++(3,6);

\end{tikzpicture}
\caption{Polygon $P^*$ associated to the quadrilateral $ABCD$}\label{polygQ}
\end{center}
\end{figure}

\subsection{Some results}
Here we explain how to find the vectors $a_1,\dots, a_n$.
\begin{definition}
For each cone $C_i$, let $d_i$ be a vector parallel to the edge $l_i$ such that
$d_i+a_i$ is colinear to $l_{i+1}$.
\end{definition}

\begin{proposition}\label{lem:G=S}
Consider the cone bounded by the lines $l_i, l_{i+1}$ 
and associated to the vector $a_i$. We have
\begin{itemize}
\item The strips associated to the lines $l_i$ and $l_{i+1}$  are consecutive for the slopes and $V_i=a_i$. 
\item The vectors $a_i, a_{i+1}$ have one vertex in common. 
\item The parallelogramm $\Sigma_i\cap \Sigma_{i+1}$ has $a_i$ for diagonal and $d_i$ for one side. The area of $\Sigma_i\cap\Sigma_{i+1}$ is equal to
 $|a_i\wedge d_i|$.
\end{itemize}
\end{proposition}
\begin{proof}
Consider the cone $C_i$ with boundaries $l_i,l_{i+1}$ and a polygon $Q\in\mathcal{S}_P$. Let $A$ be the vertex of $Q$ associated to $M$.
The first thing to remark is that the slope of the line $(AM)$ is between the slopes of $l_i$ and $l_{i+1}$. Thus $A$ belongs to the edge parallel to $l_{i+1}$ and the point $\tilde TA$ belongs to the edge parallel to $l_i$. This proves $V_i=a_i$ and the first point.

Consider one strip with vertices $A,B,M$, it means that 
$M$ is the vertex that maximized the distance from $(AB)$. 
By definition the polygon is in the strip between $(AB)$ and $M+\mathbb{R}\vec{AB}$. 
Let $N$ be the vertex neighbour of $M$ in the polygon. Assume the vertex associated to $(MN)$ is not $B$, denote it $B'$. The polygon is in the strip associated to $M,N,B'$. 
Thus this strip does not intersect the segment $[AB]$. Then the line $(MN)$ has a slope bigger than $(BB')$. First part implies that in the ordering of the slopes, the slope of $(MN)$ is the consecutive of the slope of $(AB)$, contradiction.

The vector $d_i$ is on the boundary of $\Sigma_i$ by definition. Denote $a_i=v-w$ with $v, w$ vertices of $P$. By the previous point, there exists a vertex $w'$ such that  $ww'$ is on the boundary of $\Sigma_{i+1}$. Thus one side of  $\Sigma_i\cap \Sigma_{i+1}$ is given by the line $ww'$ and one side by the line $d_i$. The area of the parallelogramm is the cross product of one side by the diagonal. 
\end{proof}

\subsection{Comments}
The preceding proposition may seem awkward, since we are not studying directly the outer billiard map to obtain results on its dynamics. Nevertheless we can transform the statement in terms of the outer billiard map $T$. The map $T^2$ is a piecewise translation, defined on several subsets of $\mathbb{R}^2$. Some of them can be compact sets, see Figure \ref{translation}. Outside a compact region containing the polygon, the sets are unbounded and the translation vectors are two by two opposite. The translation vectors are exactly the vectors $V_i$, see Proposition \ref{lem:G=S}.
The dynamics of $T^2$ is simple: A point $m$ begins its trajectory by being translated  by a vector $V_i$ until it reaches another set where it moves by another vector $V_i$. Thus an orbit of point far away from $P$ looks like  the polygon $P^*$. The link between $T^2$ and the piecewise translations of vectors $V_1\dots V_n$ can be extended to a neighborhood of $P$, but it is much more complicated, see the Pinwheel theorem \cite{Sch.11}. It is related in Proposition \ref{lem:closed} to the case where $A$ is closed to $M$. It is possible that the condition $A,\tilde TA\in C_i$ is not verified. This case is treated by the Pinwheel theorem in \cite{Sch.09}. 

\begin{figure}
\begin{center}
\begin{tikzpicture}[scale=.5]

\fill (0,0)--(2,0)--(1,2)--(0,1)--cycle;
\draw (1,2) node[left]{$C$};
\draw (0,0) node[left]{$A$};
\draw (2,0) node[below]{$B$};
\draw (0,1) node[left]{$D$};
\draw (0,0)--(8,0);
\draw (4,0)--(4,6);
\draw (5,0)--++(3,3);
\draw (2,0)--++(-3,6);
\draw (0,0)--(0,-7);
\draw (0,-4)--++(2,-4);
\draw (0,1)--++(-3,-3);
\draw (0,4)--++(-3,0);
\draw (0,-1)--(1,0);
\draw (0,-2)--(-1,0);
\draw (3,4) node[below]{$2\vec{BA}$};
\draw (5,5) node[below]{$2\vec{BD}$};
\draw (8,2) node[below]{$2\vec{BC}$};
\draw (3,-3) node[below]{$2\vec{AC}$};
\draw (0.5,-7) node[below]{2$\vec{AB}$};
\draw (-1,-3) node[below]{2$\vec{DB}$};
\draw (-3,2) node[below]{$2\vec{CB}$};
\draw (-3,6) node[below]{$2\vec{CA}$};
\end{tikzpicture}
\end{center}
\caption{Definition of $T^2$}\label{translation}
\end{figure}
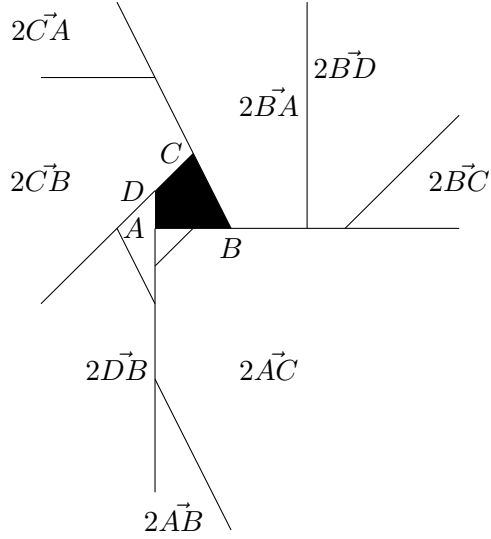

\section{Equivalence}\label{sec:equivalence}
\subsection{Statement of results}
The aim is to prove 
\begin{theorem}\label{thm:equiv}
The followings are equivalent:
\begin{itemize}
\item  $(\frac{v^*_1}{|a_1|},\dots,\frac{v_n^*}{|a_n|})$ are in $\mathbb{PQ}^n$.
\item $(|a_1\wedge a_{2}|,\dots ,|a_n\wedge a_1|)$ are commensurates.
\item $(|\Sigma_2\cap \Sigma_{1}|,\dots,|\Sigma_n\cap \Sigma_{1}|)$ are commensurates.
\end{itemize}
\end{theorem}

\begin{remark}
The first point is the initial definition of a quasi-rational polygon given in \cite{Gu.Si.92}. The third is the definition by Schwartz in \cite{Sch.09}.
\end{remark}
To do this we will prove the three following propositions. The theorem will be a clear consequence with help of Proposition \ref{lem:G=S}.
\begin{proposition}\label{prop:equiv1}
The followings are equivalent:
\begin{itemize}
\item There exists a rational solution $(t_1,\dots, t_n)$ to
$$\begin{cases}
d_1+a_1=t_2d_2\\
d_1+a_1+t_2a_2=t_3d_3\\
d_1+a_1+\dots t_na_n=-d_1
\end{cases}$$
\item $(|a_1\wedge d_1|,\dots,|a_n\wedge d_n|)$ are commensurates.
\end{itemize}
\end{proposition}

\begin{proposition}\label{prop:equiv2}
The followings are equivalent:
\begin{itemize}
\item There exists a rational solution $(t_1,\dots, t_n)$ to
$$\begin{cases}
d_1+a_1=t_2d_2\\
d_1+a_1+t_2a_2=t_3d_3\\
d_1+a_1+\dots t_na_n=-d_1
\end{cases}$$
\item $P$ is a quasi-rational polygon.
\end{itemize}
\end{proposition}
 
\begin{proposition}\label{prop:equiv3}
The followings are equivalent:
\begin{itemize}
\item There exists a rational solution $(t_1,\dots, t_n)$ to 
$$\begin{cases}
d_1+a_1=t_2d_2\\
d_1+a_1+t_2a_2=t_3d_3\\
d_1+a_1+\dots t_na_n=-d_1
\end{cases}$$
\item $(\frac{v_1^*}{|a_1|},\dots,\frac{v_n^*}{|a_n|})$ are in $\mathbb{PQ}^n$.

\end{itemize}
\end{proposition}

\subsection{Proof of Proposition \ref{prop:equiv1}}
The proof is based on Figure \ref{fig:cruc}. Consider a polygon such that the points $A,B,C,D,E$ are vertices labelled in such a way that the slopes of edges are in the increasing order $AD, BC, BE$. 
Also assume we have  $a_1=2\vec{DC}$. 
Then Proposition \ref{lem:G=S} implies that $a_2=2\vec{BD}$. 
Let $G$ be the intersection point of $(AD)$ and $(BC)$, and let $H$ a point on the line $(BC)$ such that $(HD)$ is parallel to $(BE)$. 
Then we have $d_1=2\vec{GD}, d_2=2\vec{HB}$. Moreover $\Sigma_1\cap \Sigma_2$ is defined by the triangle $GCD$, and  $\Sigma_3\cap \Sigma_2$ is defined by $BDH$.
Let $r$ be the real number such that $\vec{BH}=r\vec{GC}$. 
We see that $\Sigma_1$ is given by $((AD), 2\vec{DC})$, $\Sigma_2=((BC), 2\vec{BD})$. The intersection of the two strips $\Sigma_1, \Sigma_2$ has $DC$ as a diagonal. A similar computation gives the intersection of the strips $\Sigma_3, \Sigma_2$. The two parallelograms are constructed on triangles $BHD, GCD$. We have $$|BHD|=|\vec{BH}\wedge \vec{BD}|$$
$$|GCD|=|\vec{GC}\wedge \vec{CD}|=|\vec{GC}\wedge\vec{BD}|=r|BHD|.$$
The ratio of the areas of the two parallelograms is the same as the ratio of the area of triangles, thus it is equal to $r$.
Thus we have proved:
$$r\in \mathbb{Q}\Longleftrightarrow \frac{|\Sigma_1\cap \Sigma_2|}{|\Sigma_2\cap\Sigma_3|}\in \mathbb{Q}.$$
Note that $r$ is equal to the inverse of $t_2$ in the first part of Proposition. 
The proof of Proposition follows by induction.
\begin{figure}
\begin{center}
\begin{tikzpicture}
\draw (0.5,0)--(2,0);
\draw (2,0)--(1,2);
\draw (1,2)--(0,1);
\draw (0,1)--++(-0.125,-0.5);
\draw[dashed] (-1,0)--(0,1);
\draw[dashed] (-1,0)--(0.5,0);
\draw[dashed] (3,4)--(2,0);
\draw[dashed] (3,4)--(1,2);
\draw (0.5,0) node{$\bullet$} node[below]{$A$};
\draw (2,0) node{$\bullet$} node[below]{$D$}; 
\draw (0,1) node{$\bullet$} node[left]{$ B$};
\draw (1,2) node{$\bullet$} node[above]{$ C$};
\draw (-1,0) node{$\bullet$} node[left]{$G$};
\draw (3,4) node{$\bullet$} node[right]{$H$};
\draw (-0.125,0.5)  node{$\bullet$} node[right]{$E$};
\end{tikzpicture}
\begin{tikzpicture}[scale=0.5]
\draw (2,0) node[below]{$D$}; 
\draw (0,1) node[left]{$B$};
\draw (3,4) node[right]{$H$};
\draw (-1,0)--(4,5);
\draw (0,1)--(2,0);
\draw (1,-4)--(3.5,6);
\draw (-1,-3)--(4,2);
\draw (-1,-3)--(1,5);
\draw (2,4) node[above]{$\Sigma_3$};
\draw (3,2.4) node[right]{$\Sigma_2$};
\fill [gray] (0,1)--(3,4)--(2,0)--(-1,-3);
\end{tikzpicture}
\caption{Proof of Proposition \ref{prop:equiv1} }\label{fig:cruc}
\end{center}
\end{figure}
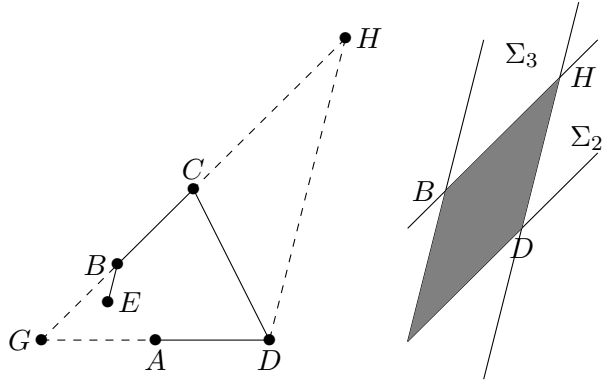

\subsection{Proof of Proposition \ref{prop:equiv2}}
The proof is based on Figure \ref{fig:cones}.
Consider an edge $AD$, and the associated vector $a_1=2\vec{DC}$. By Proposition \ref{lem:G=S} the vector $a_2$ is equal to $2\vec{BD}$ with $BC$ edge of the polygon, and $a_3=2\vec{FB}$, with $DF$ edge of the polygon. 
Let us call $G$ the intersection of $(CB)$ with $(AD)$, and $H$ the intersection of $(CB)$ and $(DF)$. Then $d_1=2\vec{GD}, d_2=2\vec{HB}$. 
Assume that the first item of Proposition \ref{prop:equiv2} holds. Then there exist $r,r'\in\mathbb{Q}$ such that $\vec{GC}=r\vec{HB}, \vec{HD}=r'\vec{DF}$. Solving system shows that $d_i$ is a rational linear sum of $a_1,\dots, a_n$.  By Proposition \ref{lem:G=S}, the edges of $P$ are rational combination of $a_1,\dots, a_n$, the assumption implies: $\vec{GC}=q\vec{BC}, \vec{HD}=q'\vec{FD}$ with $q,q'\in \mathbb{Q}$. Now the relations $\vec{GC}=q\vec{BC}=r\vec{HB}$ imply $\vec{HB}=q"\vec{CB}$ with $q''$ rational number. 
For a point $M$, denote $h_M$ the length of the orthogonal projection of $M$ on $(DB)$. The relations $\vec{HB}=q"\vec{CB}, \vec{HD}=q\vec{FD}$ gives  $h_H=q" h_C$ and $h_F=q h_H$.
By Proposition \ref{lem:G=S}, the areas $|a_1\wedge a_2|, |a_2\wedge a_3|$ are given by areas of triangles $BCD, DBF$. The ratio of these areas is  equal to the ratio between $h_C$ and $h_F$. Thus the areas are commensurates. The other part of the proof is similar.

 \begin{figure}[h]
 \begin{center}
\begin{tikzpicture}
\draw (0,0)--(2,0);
\draw (2,0)--(2.5,1.5);
\draw (2.5,1.5)--(1,2);
\draw (1,2)--(0,1);
\draw (0,1)--(0,0);
\draw (4,-1)--(-2,2);
\draw[dashed] (-1,0)--(0,1);
\draw[dashed] (-1,0)--(0,0);
\draw[dashed] (0,1)--(4,5);
\draw[dashed] (3,4)--(1,2);
\draw[dashed] (2,0)--(4,6);
\draw[dashed] (-1,0)--++(0.6,1.2);
\draw[dashed] (3.5,4.5)--++(-2.1,-4.2);
\draw (0,0) node[below]{$A$};
\draw (2,0) node[below]{$D$}; 
\draw (0,1) node[above]{$ B$};
\draw (1,2) node[above]{$ C$};
\draw (-1,0) node[left]{$G$};
\draw (3.5,4.5) node[above]{$H$};
\draw (2.5,1.5) node[right]{$F$};
\end{tikzpicture}
\caption{Proof of Proposition \ref{prop:equiv2}}\label{fig:cones}
\end{center}
\end{figure}

\subsection{Proof of Proposition \ref{prop:equiv3}}
The proof is based on Figure \ref{polygQ}. 
By Proposition \ref{prop:equiv2} we know that the first statement is equivalent to the fact that $P$ is a quasi-rational polygon. If the system has a rational solution, then the polygon defined by $d_1,d_1+a_1,\dots, d_1+a_1+\dots t_na_n$ is some polygon $P^*$. The edges $v_1^*,\dots, v_n^*$ of this polygon are equal to $a_1, t_2a_2,\dots, t_na_n$, thus the first implication is proved.
  
Conversly, consider the point $M$ on $l_1$ such that $\vec{OM}=d_1$. By hypothesis, there exists a polygon with sides $r_1a_1, \dots r_na_n$ with rational numbers $r_i, i\leq n$.  Thus there exists an homothetic image of this polygon with vertex $M$, and all the edges fulfilling the same condition. This gives a rational solution of the system.

This proposition can be reformulated in
\begin{corollary}\label{lem:qr-equiv}
The polygon is quasi-rational if and only if :
there exists rational numbers $t_2,\dots,t_n$ and $M\in\mathbb{R}^2$ such that
$$
\begin{cases}
M\in l_1\\ 
M+a_1\in l_2\\ M+a_1+t_2a_2\in l_2\\ 
M+a_1+\dots+t_ia_i\in l_{i}\\ 
M+a_1+\dots+t_na_n=-M
\end{cases}$$
\end{corollary}

\section{All orbits are bounded for quasi-rational polygons}\label{sec:gs}
In this section we give a new proof of the following theorem using our definition of quasi-rational polygon. The aim is to understand the general outline of the proof, not to explain all the details.

\begin{theorem}\cite{Gu.Si.92}
For a quasi-rational polygon, every orbit of the dual billiard map is bounded.
\end{theorem}

We consider the unfolding and the cone $C_1$. We can tile periodically this cone by a parallelogramm with one side equal to $d_1$ and one diagonal equal to $a_1$. The same thing can be done in all cones. Consider a point $x$ and the first hitting map with the next cone: $f_1$. We have $f_1(x+d_1)=f_1(x)+a_1+d_1$, we deduce $f_2(f_1(x+d_1))=f_2(f_1(x)+t_2d_2)$.
Since the polygon is quasi-rational there exists an integer $n$ such that $f_2(f_1(x+nd_1))=f_2(f_1(x)+nt_2d_2)=f_2(x)+n'(a_2+d_2)$.
Now the first return map to the cone $C_1$ is the composition of $f_1,\dots, f_n$. We obtain that there exists a vector $u$ such that for every $x$
$$F(x+ u)=F(x)+u$$
In term of parallelograms, it means that we consider a point in one box and take the image of the box by $F$. We have a second periodic tiling of the cone by a parallelogram with side $u$ and diagonal $a_1$. The orbit of the point $x$ depends on the cutting of a box of the new tiling by the initial one. If the two tilings are commensurates then every orbit is bounded. We must compare $u$ and $d_1$: they are rationally proportional by definition of quasi-rational polygon. 
If $P$ is rational every box is mapped by $F$ to a box, thus every orbit is periodic.  

\begin{figure}[h]
\begin{center}
\begin{tikzpicture}[scale=.3]
\draw (0,0)--(12,0);
\draw (0,0)--++(10,10);
\draw (6,0)--++(-2,4);
\draw (12,0)--++(-4,8);
\draw (4,4)--++(8,0);
\draw (8,8)--++(4,0);

\draw (0,0)--(0,17);
\draw (4,4)--++(-4,2);
\draw (8,8)--++(-8,4);
\draw (0,6)--++(8,8);
\draw (0,12)--++(4,4);

\draw (0,0)--++(-9,18);
\draw (0,8)--++(-4,0);
\draw (0,16)--++(-8,0);
\draw (-4,8)--++(0,10);
\end{tikzpicture}
\end{center}
\caption{Tilings of consecutive cones}
\end{figure}
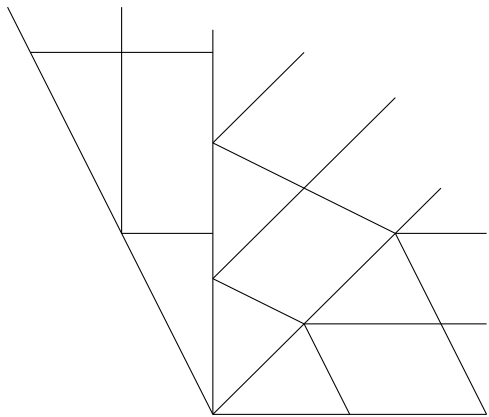

\section{Graph of spokes}
\subsection{Definitions}
By definition, for each integer $i$, $a_i$ is a vector between two vertices of $P$, and every vertex is a starting point of some $a_i$.
Define a graph with vertices the vertices of $P$, and there is an oriented edge starting from each vertex and joining the end of the vector $a_i$ associated to the vertex. Denote it $\mathcal{S}(P)$, and we call it the {\bf graph of spokes}.

\begin{example}
Consider the polygon $ABCD$ of Figure \ref{fig-def-1}, then $\mathcal{S}(P)$ is 
given by 
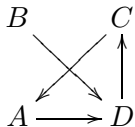
\begin{figure}[h]
\begin{center}
$\xymatrix{
B \ar[dr] & C\ar[dl] \\
A\ar[r]&D\ar[u]
}$
\caption{Graph of spokes}\label{fig:spoke}
\end{center}
\end{figure}
\end{example}

 \begin{lemma}
 This graph has following properties: 
 \begin{itemize}
\item Each vertex has an outgoing edge.
\item Two edges can not have the same vertices.
\item The graph contains a cycle.
\end{itemize}
 \end{lemma}
 \begin{proof}
 Proof left to the reader for the two first items.
 An edge between vertices $A, B$ implies that some vector $a_i=\vec{AB}$. Thus the graph is the same thing as a map defined on the set of vertices. This map is defined everywhere but not necessarily injective. It is injective on a subset. On this subset the graph is a cycle.
 \end{proof}
 
 \begin{corollary}\label{rel-alg-ai}
 For every polygon, 
there exists a rational relation between the vectors $a_1,\dots,a_n$.
\end{corollary}
\begin{proof}
By preceding Lemma there exists a cycle in the graph. It implies that the sum of vectors $a_i$ associated to this cycle is null.
\end{proof}

\begin{remark}
The notion of spokes is used in outer billiard by Schwartz to prove its result on the pinwheel map, see \cite{Sch.10}.
\end{remark}

We now use the preceding tools to obtain new results on quasi-rational polygons.
\section{Description of quasi-rational polygons}\label{sec:new}
\begin{theorem}\label{defs-quasi}
We have:
\begin{itemize}
\item A quadrilateral is a quasi-rational polygon if and only if it is rational.
\item There exists a  non regular and non rational quasi-rational pentagon.
\item Assume the graph of spokes is a cycle (or an union of cycles). Then the polygon is quasi-rational.
 \end{itemize}
\end{theorem}
\begin{proof}
\begin{itemize}
\item
Consider $(a_1, a_2)$ as a basis of $\mathbb{R}^2$, denote $\alpha, \beta$ the coordinates of $a_3$ in this basis, and $(\gamma, \delta)$ those of $a_4$:
$a_3=\alpha a_1+\beta a_2, a_4=\gamma a_1+\delta a_2$.
The numbers $|a_1\wedge a_2|, |a_3\wedge a_2|, |a_4\wedge a_3|, |a_4\wedge a_1|$ are proportional to $1, \alpha, \alpha\delta-\beta\gamma,\delta$. If the polygon is quasi-rational, by Theorem \ref{thm:equiv} we deduce $\alpha,\delta,\alpha\delta-\beta\gamma\in\mathbb{Q}$. This implies $\beta\gamma\in\mathbb{Q}$. Now Corollary \ref{rel-alg-ai} implies that there exists a rational linear relation between $a_1,\dots, a_4$. This relation concerns at least three vectors. All possibilities imply $\beta, \gamma\in\mathbb{Q}$. Thus the polygon has vertices on a lattice.

\item Consider four points on a lattice of $\mathbb{R}^2$. Denote these points $A, B, C, E$. We will contruct a point $D$ such that the pentagon $ABCDE$ will be as required. It suffices to consider one point $D$ outside the lattice. We can always choose $D$ such that  the spokes of the pentagon $ABCDE$ are associated to vectors $\vec{AC}, \vec{BE}, \vec{CE}, \vec{DB}, \vec{EA}$. Then the rational relation is $a_1+a_3+a_5=0$. There is no other rational relation by definition of $D$. 
Now we can express the vectors $a_1, \dots, a_5$ in the basis $(a_1, a_2)$. By construction $a_3, a_5$ have rational coordinates. Then we can always choose $D$ such that the area $|a_2\wedge a_3|$ is rational. The constructed pentagon is quasi-rational.

\item Now assume that the graph of spokes is an union of cycles.
By Corollary \ref{lem:qr-equiv} a polygon is quasi-rational if
 for every side $l_i$, there exists $\lambda\in\mathbb{R}$ and rational numbers $r_1,\dots, r_n\in\mathbb{Q}^*$ such that 
$$\lambda l_i+r_1a_1+\dots+r_na_n=0.$$

If the graph is a union of cycles, then the map defined on vertices associated to the graph of spokes is invertible. It means that each vertex is a linear combination of $a_1\dots a_n$. Thus the side $l_i$ can be expressed as rational combination of the $a_i$'s. Thus $P$ is quasi-rational if there exists $r_1\dots r_n\in\mathbb{Q}$ and $\lambda\in\mathbb{R}$ such that:  
$$\lambda \displaystyle\sum r'_ja_j+r_1a_1+\dots+r_na_n=0.$$
Since the graph is a cycle, there exists a rational relation between $a_1\dots a_n$. Thus we can solve the equation and find $r_1\dots r_n, r'_1\dots r'_n$.
\end{itemize}
\end{proof}

\begin{remark}
Consider the example of graph in Figure \ref{fig:spoke}. In this case the preceding map is not a bijection since no edge goes to $B$.

For regular polygon with odd number of sides (greater than five), the graph is not simply connected.
\end{remark}

\section{Remarks}
\subsection{Polygon with parallel sides}
If the polygon has parallel sides, then the definition of \cite{Gu.Si.92} still works. Nevertheless the number of cones decreases. For the definition of \cite{Sch.09} we need to be more precise to define a strip. In this case two consecutive strips can have an intersection with infinite area. Thus the new definition of quasi-rational is that, up to a factor, the areas of $\Sigma_i\cap \Sigma_{i+1}$ are in $\mathbb{Z}\cup \{\infty\}$ for every integer $i$. 

\subsection{Regular polygons}
A regular polygon with $n$ edges is invariant by rotation of angle $2\pi/n$. Let $\omega$ be a $n$ th root of unity, we have $a_i=\omega a_{i-1}+a_{i-2}$ for every integer $i$. Thus it is clear that $|a_i\wedge a_{i+1}|$ is a constant number, and a regular polygon is a quasi-rational polygon. Moreover the graph of spokes is a cycle, since the spoke $a_{i+1}$ is the image of $a_i$ by rotation of angle $2\pi/n$. This gives another proof of previous fact.

The study of regular polygons has been done if the number of sides is equal to $5$ by Tabachnikov, see \cite{Ta.95}. A description of the symbolic dynamics has been made for regular polygons with $3,4,5,6, 10$ edges in \cite{Bed.Ca.10}. In \cite{Sch.10} Schwartz initiates a study of the regular octogon.

\bibliographystyle{alpha}
\bibliography{bibli-quasirat}

\begin{thebibliography}{Tab95b}

\bibitem[BC11]{Bed.Ca.10}
N.~Bedaride and J.~Cassaigne.
\newblock Outer billiards outside regular polygons.
\newblock {\em Journal of London Mathematical Society}, 2(83):301--323, 2011.

\bibitem[Gen08]{Gen.08}
D.~Genin.
\newblock Research announcement: boundedness of orbits for trapezoidal outer
  billiards.
\newblock {\em Electron. Res. Announc. Math. Sci.}, 15:71--78, 2008.

\bibitem[GS92]{Gu.Si.92}
E.~Gutkin and N.~Sim{\'a}nyi.
\newblock Dual polygonal billiards and necklace dynamics.
\newblock {\em Comm. Math. Phys.}, 143(3):431--449, 1992.

\bibitem[Ko{\l}89]{Kol}
R.~Ko{\l}odziej.
\newblock The antibilliard outside a polygon.
\newblock {\em Bull. Polish Acad. Sci. Math.}, 37(1-6):163--168 (1990), 1989.

\bibitem[Li09]{Li.09}
L.~Li.
\newblock On {M}oser's boundedness problem of dual billiards.
\newblock {\em Ergodic Theory Dynam. Systems}, 29(2):613--635, 2009.

\bibitem[Neu59]{Neu.59}
B.H.R. Neumann.
\newblock Sharing ham and eggs.
\newblock {\em Iota, Manchester university Mathematics students journal}, 1959.

\bibitem[Sch07]{Sch.07}
R.~E. Schwartz.
\newblock Unbounded orbits for outer billiards. {I}.
\newblock {\em J. Mod. Dyn.}, 1(3):371--424, 2007.

\bibitem[Sch09]{Sch.09}
R.~E. Schwartz.
\newblock {\em Outer billiards on kites}, volume 171 of {\em Annals of
  Mathematics Studies}.
\newblock Princeton University Press, Princeton, NJ, 2009.

\bibitem[Sch10]{Sch.10}
R.~E. Schwartz.
\newblock Outer billiards, the arithmetic graph and the octagon.
\newblock {\em Arxiv}, 2010.

\bibitem[Sch11]{Sch.11}
Richard~Evan Schwartz.
\newblock Outer billiards and the pinwheel map.
\newblock {\em J. Mod. Dyn.}, 5(2):255--283, 2011.

\bibitem[Tab95a]{Ta}
S.~Tabachnikov.
\newblock Billiards.
\newblock {\em Panoramas et Synth\`eses}, 1995.

\bibitem[Tab95b]{Ta.95}
S.~Tabachnikov.
\newblock On the dual billiard problem.
\newblock {\em Adv. Math.}, 115(2):221--249, 1995.

\bibitem[VS87]{Viv.Shai.87}
F.~Vivaldi and A.~V. Shaidenko.
\newblock Global stability of a class of discontinuous dual billiards.
\newblock {\em Comm. Math. Phys.}, 110(4):625--640, 1987.

\end{thebibliography}
\end{document}